\documentclass[a4paper,11pt]{article}

\input{styleart.sty}

\title{Forking and JSJ decompositions in the free group II}
\date{\today}
\author{Chlo\'{e} Perin \and Rizos Sklinos}

\begin{document}

\maketitle

\begin{abstract} 
We give a complete characterization of the forking independence relation over any set of parameters in the free groups of finite rank, in terms of the $JSJ$ decompositions relative to those parameters.
\end{abstract}

\section{Introduction} 

In \cite{SelaStability}, Sela proves that the first order theory of torsion free hyperbolic groups, and thus in particular that of free groups, is stable. This amazing result implies the existence of a good notion of independence between tuples of elements over any set of parameters, akin to that of algebraic independence in algebraically closed fields. It was thus natural to ask whether there was a way to characterize this notion of independence (called forking independence) in a purely group theoretic way.

In \cite{PerinSklinosForking} we gave such a description for two families of parameter sets: those which are free factors of the free group $\F$, and at the other extreme, those which are not contained in any proper free factor of $\F$.

Under the assumption that the parameter set $A$ is a free factor of $\F$, we showed two tuples $b$ and $c$ are independent over $A$ if and only if there is a free product decomposition $\F = F_b * A * F_c$ such that $b \in F_b * A$ and $c \in A*F_c$. In other words, the tuples are independent if and only if there is a Grushko decomposition for $\F$ relative to $A$ for which $b$ and $c$ live in "different parts".

In the case where $A$ is not contained in any proper free factor, we proved that two tuples $b$ and $c$ are independent over $A$ if and only if they live in "different parts" of the pointed cyclic JSJ decomposition of $\F$ relative to $A$. A cyclic JSJ decomposition relative to $A$ is a graph of groups decomposition which encodes all the splittings of $\F$ as an amalgamated product or an HNN extension over a cyclic group for which $A$ is contained in one of the factors. Under our assumption on $A$, there is a canonical JSJ decomposition for $\F$ relative to $A$, namely the tree of cylinders of the deformation space. 

The aim of the present paper is to complete this description for any set of parameters. In this setting, there is no canonical JSJ decomposition - we thus give a condition in terms of all the \textbf{normalized} pointed cyclic JSJ decompositions (see Definition \ref{NormalizedJSJDef}). Our main result is:
\begin{thmIntro}\label{MainTheorem}
Let $A\subset \F$ be a set of parameters and $b,c$ be tuples from $\F$. Then $b$ is 
independent from $c$ over $A$ if and only if there exists a normalized cyclic $JSJ$ decomposition $\Lambda$ of $\F$ relative to $A$ in which any two blocks of the minimal subgraphs $\Lambda^{min}_{Ab}, \Lambda^{min}_{Ac}$ of $\langle A,b\rangle$ 
and $\langle A,c\rangle$ respectively intersect at most in a disjoint union of envelopes of rigid vertices. 
\end{thmIntro}   
(For a precise definition of blocks of minimal subgraphs and envelope of a rigid vertex, see Definition \ref{EnvelopesDef}).

Before Sela's work proving the stability of torsion free hyperbolic groups, only the families of abelian groups and algebraic groups (over algebraically closed fields) were known to be stable. For these families it was fairly easy to understand and characterize forking independence. But there is a qualitative difference between the already known examples of 
stable groups and torsion-free hyperbolic groups: a group elementarily equivalent to an abelian group is an 
abelian group and likewise a group elementarily equivalent to an algebraic group is an algebraic group. 
This is certainly false in the case of torsion-free hyperbolic groups: an ultrapower of a non abelian free 
group under a non-principal ultrafilter is {\bf not} free. Even more, by definition, a non finitely generated 
group cannot be hyperbolic.  

The difficulty in describing forking independence arising from this easy observation is significant. To emphasize this fact we recall the definition of forking independence. 
\begin{definition} Let $M$ be a stable structure and let $A$ be a subset of $M$. Tuples $b$ and $c$ of elements of $M$ fork over $A$ (in other words, are NOT independent over $A$) if and only if there exists a set $X$ definable over $Ac$ which contains $b$, and a sequence of automorphisms $\theta_n \in \Aut_A(\hat{M})$ for some elementary extension $\hat{M}$ of $M$, such that the translates $\theta_n(X)$ are $k$-wise disjoint for some $k \in \N$.
\end{definition}
(For some intuition on the notion of forking, see Section 2 of \cite{LouPeSk}). 

Thus a priori, even for understanding the independence relation provided by stability in natural models of 
our theory, one has to move to a {\em saturated model} of the theory. The main problem with that is that 
we have very little knowledge of what a saturated model of the theory of nonabelian free groups look like. 
In \cite{PerinSklinosForking} we managed to overcome this difficulty by using the assumptions we imposed on the parameter set $A$. When $A$ is not contained in any proper free factor, there are a number of useful results available. Model theoretically: under this assumption, $\F$ is atomic over $A$, i.e. every type can in fact be defined by a single formula (see \cite{PerinSklinosForking}, \cite{OuldHoucineHomogeneity}). This enables us on the one hand to transfer a sequence witnessing forking from a big model to $\F$, but more importantly it gives us a natural candidate for a formula witnessing forking: by homogeneity (see \cite{PerinSklinosHomogeneity} and \cite{OuldHoucineHomogeneity}), types in $\F$ correspond to orbits under the automorphism group, so in fact the orbit of a tuple under $\Aut_A(\F)$ is definable. Geometrically, under the assumption that $A$ is not contained in any proper free factor, 
we have a very good understanding of $\Aut_A(\F)$: the canonical JSJ decomposition of $\F$ relative to $A$ enables one to describe up to finite index the automorphisms fixing $A$ (one understands the modular automorpshism group $\Mod_A(\F)$). 

The following example illustrates why in Theorem \ref{MainTheorem}, it is necessary to consider all JSJ decompositions with respect to $A$. The idea is that with trivially stabilized edges one can create "artificial intersection" between the minimal subgraphs of two tuples.

\begin{example} Suppose that $\F = \F_A * \langle t \rangle$, where $\F_A$ is the smallest free factor of $\F$ containing $A$. Suppose further that the cyclic JSJ decomposition $\Lambda_A$ of $\F_A$ relative to $A$ (corresponding to the tree of cylinders) consists of two rigid vertices $u,v$ with stabilizers $U$ and $V$, where $A \leq U$, and one $Z$-type vertex $z$ with cyclic stabilizer, and two edges joining $z$ to $u,v$ respectively (see Figure \ref{FigTwoJSJs}).

\begin{figure*}[ht!]\label{FigTwoJSJs}
\centering
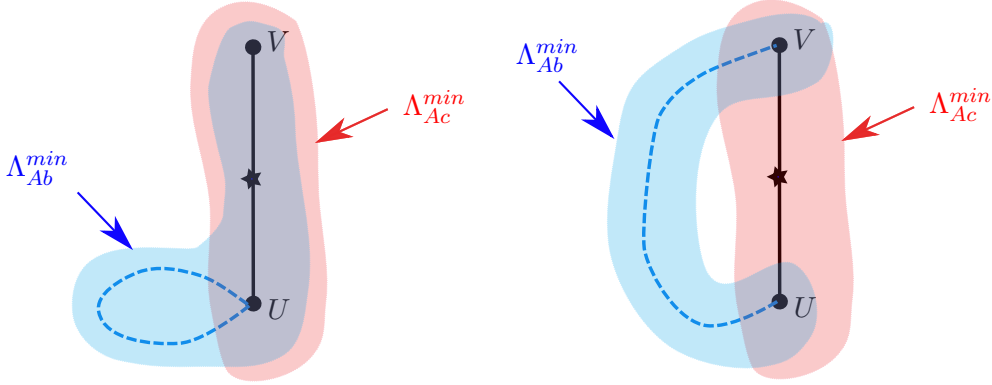
\caption{In the first JSJ decomposition, the minimal subgraphs of $Ab$ and $Ac$ meet in more than a union of envelopes of a rigid vertex, while in the second their intersection contains only rigid vertices. (Dot vertices represent rigid vertices, stars represent $Z$-type vertices, dashed edges are trivially stabilized while full ones have infinite cyclic stabilizers.)}
\end{figure*}

Let $g \in V$, and consider the elements $c=g$ and $b=tgt^{-1}$. 

A normalized JSJ decomposition for $\F$ consists of adding exactly one trivially stabilized edge to $\Lambda_A$, with Bass-Serre element corresponding to $t$ (see Figure \ref{FigTwoJSJs}). If we add it as a loop joining $u$ to itself, then the minimal subgraph of $\langle A, tgt^{-1} \rangle$ will be the whole graph of groups $\Lambda$. Since the minimal subgraph for $\langle A, g \rangle$ is $\Lambda_A$, the two minimal subgraph intersect in two distinct edges of a cylinder - this is not contained in a disjoint union of envelopes of rigid vertices (see Definition \ref{EnvelopesDef}). On the other hand, if we add it as an edge joining $u$ to $v$, the minimal subgraph of $\langle A, tgt^{-1} \rangle$ will "jump over" the edges of the minimal subgraph of $\langle A, g \rangle$, and the intersection will contain only the vertices $u$ and $v$, which means that the condition for independence in Theorem \ref{MainTheorem} is satisfied. 
\end{example}

\paragraph{Outline of the proof. } We start by proving the right to left direction of Theorem \ref{MainTheorem} in the special case where $b$ is contained in the smallest free factor $\F_A$ containing $A$. In this case, we can assume $\F$ to be freely indecomposable with respect to $Ac$, and thus the orbit $X =\Aut_{Ac}(\F) \cdot b$ of $b$ under $\Aut_{Ac}(\F)$ is definable over $Ac$ (see Theorem 5.3 in \cite{PerinSklinosForking}). We then proceed to show that under the assumption that the minimal subgraphs are essentially disjoint, the orbit of $b$ under the modular group $\Mod_{Ac}(\F)$, which is contained in $X = \Aut_{Ac}(\F) \cdot b$, is almost $A$-invariant, i.e. that it has finitely many images by $\Aut_A(\F)$. We do this by showing that this set is $\Mod_A(\F_A)$ invariant, that is, that any image of $b$ by a modular automorphism of $\F_A$ fixing $A$ can be obtained by a modular automorphism of $\F$ fixing both $A$ and $c$ - this relies mostly on the understanding of modular group given by the JSJ decompositions. Note that since $X$ is contained in any set definable over $Ac$ which contains $b$, this implies that we cannot find a sequence of automorphisms $\theta_n$ \textbf{of the standard model} $\F$ which would witness forking, since any translate of $X$ by such an automorphism would contain one of the finitely many images of $\Mod_{Ac}(\F) \cdot b$ under $\Aut_A(\F)$. Now $X$ is contained in $\F_A$, hence it is in fact definable over $\F_A$. Now according to Lemma 2.10 in \cite{PerinSklinosForking} (see Lemma \ref{AtoFork} in this paper), atomicity of $\F_A$ over $A$ implies that even passing to an elementary extension of $\F$ one cannot find such a witnessing sequence, in other words, that $b$ and $c$ are independent over $A$. 

To prove the general case of the same direction, we decompose $b$ and $c$ in what we call \textbf{sandwich terms} - that is, terms of the form $t\beta s$ where $t,s$ are Bass-Serre elements associated to trivially stabilized edges, and $\beta$ is an element of $\F_A$. We show in Proposition \ref{JoinTheDots} that if $b$ and $c$ are tuples of sandwich terms, and there exists a subgraph $\Delta$ of $\Lambda$ which contains the minimal subgraph of $Ab$, whose intersection with the minimal subgraph $\Lambda_A$ of $\F_A$ is connected, and whose intersection with the minimal subgraph of any element of $c$ is contained in a disjoint union of envelopes of rigid vertices, then $b$ and $c$ are independent over $A$. By subdividing the sandwich terms obtained from general tuples $b$ and $c$ into alternating groups according to the path that join them to $v_A$, we finally prove the general case by induction and forking calculus (see Example \ref{ForkingCalculusEx} for a special case). 

To prove the other direction, we show first that there exists for each tuple $b$ in $\F$ some subgraphs $\Gamma^i_{Ab}$ of the JSJ decomposition $\Lambda_A$ of $\F_A$ (the smallest free factor containing $A$) from which one can read off some elements of $\acl^{eq}(Ab)$, and such that there exists a normalized JSJ decomposition $\Lambda_{Ab}$ for $\F$ relative to $A$ in which the minimal subgraph of $Ab$ is contained in a union of these subgraphs with trivially stabilized edges. We first see that if $b$ and $c$ are independent over $A$ the subgraphs $\Gamma^i_{Ab}$ and $\Gamma^j_{Ac}$ must intersect in disjoint union of envelopes of rigid vertices.

From the two decomposition $\Lambda_{Ab}$ and $\Lambda_{Ac}$ we build a group $\hat{F} = \F *_{\F_A}*\F'$ which is an amalgamation of two copies of $\F$ along $\F_A$. Then by the first direction of the main result (which we already proved), we see that $b$ must be independent over $A$ from the element $c'$ corresponding to $c$ in $\F'$. By then using stationarity of types and homogeneity of the free group (following the lines of the proof of \cite[Theorem 3.4]{PerinSklinosForking} ), we show that there exists a decomposition of $\F$ as $\F_b*\F_A*\F_c$ such that $Ab \in \F_b * \F_A$ and $Ac \in \F_A * \F_c$. It then follows that the decompositions $\Lambda_{Ab}$ and $\Lambda_{Ac}$ can be combined into a normalized JSJ for $\F$ which satisfies the necessary conditions. 

\paragraph{Acknowledgements.} We are grateful to Zlil Sela for pointing out a mistake in the original version of Proposition 5.3.

\section{Stability} \label{StabilitySec}

A detailed account of the background on stability needed for this article has been already developed in 
\cite[Section 2]{PerinSklinosForking}, thus for avoiding repetition we only state here the notions and results of fundamental importance to this article.  

We fix a stable first-order theory $T$ and we work in a ``big'' saturated model $\mathbb{M}$ of $T$, which  
is usually called the {\em monster model} (see \cite[p.218]{MarkerModelTheory}). Capital letters $A,B,C,\ldots$ denote 
parameter sets of small cardinality, i.e. cardinality strictly less than $\abs{\mathbb{M}}$.

\begin{definition} A formula {\em $\phi(\bar{x},b)$ forks over $A$} if there are $n<\omega$ and 
an infinite sequence $(b_i)_{i<\omega}$ of tuples in $\mathbb{M}$ such that $tp(b/A)=tp(b_i/A)$ for $i<\omega$, and 
the set $\{\phi(\bar{x},b_i) : i<\omega\}$ is $n$-inconsistent (i.e, any subset of $n$ formulas is inconsistent).

A tuple $\bar{a}$ is {\em independent} from $B$ over $A$ (denoted $\bar{a} \underset{A}{\forkindep} B$)
if there is no formula in $tp(\bar{a}/B)$ which forks over $A$.
\end{definition}

Forking independence satisfies certain axioms, a list of which we have recorded in \cite[Section 2]{PerinSklinosForking}. 
We give next those important for this paper. 

\begin{fact} (Symmetry) $B \forkindep_A C$ if and only if $C \forkindep_A B$.
\end{fact}

\begin{fact}  \label{ForkTrans} (Transitivity)
Let $A\subseteq B\subseteq C$. Then $\bar{a} \underset{A}{\forkindep} C$    
if and only if $\bar{a} \underset{A}{\forkindep} B$ and $\bar{a} \underset{B}{\forkindep} C$.
\end{fact}

\begin{fact}\label{ForkAlg}
Let $A\subseteq B$. Then $\bar{a} \underset{A}{\forkindep} B$ if and only if $acl(\bar{a}A) \underset{acl(A)}{\forkindep} acl(B)$. 
\end{fact}

In the above fact and for the rest of the article the algebraic closure $acl$ is 
understood as the imaginary algebraic closure $acl^{eq}$, that is, it is the algebraic closure in the extended theory $T^{eq}$, where we have added one sort for each definable equivalence relation, so that equivalence classes can be thought of as real elements. See \cite[page 10]{PillayStability} for more details.

\begin{definition} Let $A \subseteq B$ be sets of parameter, and let $p \in S_n(A)$. A \textbf{non forking extension} of $p$ over $B$ is a complete type $q \in S_n(B)$ containing $p$ such that no formula $\phi(\bar{x},b) \in q$ forks over $A$. A type $p$ over $A$ is \textbf{stationary} if for any $A\subset B$, the type $p$ admits a unique non-forking extension over $B$.  
\end{definition}

\begin{fact}\label{AlgStat}
Every type over an algebraically closed set $A$, i.e. $A=acl(A)$, is stationary.
\end{fact}

The following lemma is \cite[Lemma 2.10]{PerinSklinosForking}.
\begin{lemma}\label{AtoFork}
Let $\mathcal{M}\models T$. Let $b,A\subset \mathcal{M}$, 
and $\mathcal{M}$ be countable and atomic over $A$. Suppose that the set $X\subset \mathcal{M}$ defined by a formula $\phi(\bar{x},b)$ contains 
a non empty almost $A$-invariant subset (i.e. a subset that has finitely many images under $Aut(\mathcal{M}/A)$). 
Then $\phi(\bar{x},b)$ does not fork over $A$.   
\end{lemma}

Finally, we will use the following fact
\begin{fact}\label{ExtDef} 
Let $\mathcal{M}\prec\mathbb{M}$. Let $X\subset\mathbb{M}$ be a set definable over $\mathbb{M}$. 
Then $X\cap \mathcal{M}$ is definable over $\mathcal{M}$.
\end{fact}

\section{Trees} \label{TreesSec}

For this section we use the notations and framework of \cite{GuirardelLevittUltimateJSJ}.

Let $G$ be a finitely generated group. A $G$-tree is a simplicial tree $T$ together with an action of $G$ on $T$. A cyclic $G$-tree is a $G$-tree whose edge stabilizers are cyclic (possibly finite or trivial). We will only consider cyclic $G$-trees in this paper.

We say a $G$-tree $T$ is {\em minimal} if it admits no proper $G$-invariant subtree. Note that unless otherwise mentioned, the $G$-trees we will consider are \textbf{not} minimal.  

If $H$ is a finitely generated subgroup of $G$, it is a standard result that there is a unique subtree of $T$ preserved by $H$ on which the action of $H$ is minimal. We denote it by $T^{min}_H$.
\begin{definition} \label{MinimalSubgraphDef} Let $T$ be a $G$-tree, with corresponding graph of groups $\Lambda$, and let $H$ be a finitely generated subgroup of $G$. 
The minimal subgraph $\Lambda^{min}_{H}$ of $H$ in $\Lambda$ is the image of $T^{min}_H$ under the quotient map. 
\end{definition}

We recall from \cite{PerinSklinosForking} the definition of a Bass-Serre presentation:
\begin{definition} \label{BassSerrePresentationDef}
\emph{(Bass-Serre presentation)} Let $G$ be a finitely generated group, and let $T$ be a $G$-tree. Denote by $\Lambda$ the 
corresponding quotient graph of groups and by $p$ the quotient map $T \to \Lambda$. 

A Bass-Serre presentation for $\Lambda$ is a triple $(T^1, T^0, (t_e)_{e  \in E_1)})$ consisting of
\begin{itemize}
\item a subtree $T^1$ of $T$ which contains exactly one edge of $p^{-1}(e)$ for each edge $e$ of $\Lambda$;
\item a subtree $T^0$ of  $T^1$ which contains exactly one vertex of $p^{-1}(v)$ for each vertex $v$ of $\Lambda$;
\item for each edge $e \in E_1:= \{ e=uv \mid u \in T^0, v \in T^1\setminus T^0 \}$, an element $t_e$ of $G$ such that $t_e^{-1} \cdot v$ 
lies in $T^0$.
\end{itemize}
We call $t_e$ the stable letter associated to $e$. 
\end{definition}

If $A$ is a subset of $G$, a {\em $(G,A)$-tree} is a $G$-tree in which $A$ fixes a point. 
A (not necessarily simplicial) surjective equivariant 
map $d: T_1 \to T_2$ between two $(G,A)$-trees is called a {\em domination map}. It sends edges to paths in $T_2$. We then say that $T_1$ {\em dominates} $T_2$. This is equivalent to saying that any subgroup of $G$ which is elliptic in $T_1$ is also elliptic in $T_2$.
  
A surjective simplicial map $p:T_1 \to T_2$ which consists in collapsing some orbits of edges to points 
is called a {\em collapse map}. In this case, we also say that $T_1$ {\em refines} $T_2$.

\subsection{Cylinders and envelopes}

\begin{definition} Let $T$ be a $G$-tree with infinite cyclic edge stabilizers. A \textbf{cylinder} in $T$ is an equivalence class of edges under the equivalence relation given by $e \sim f$ iff $\Stab(e)$ and $\Stab(f)$ are commensurable (see Example (3) of Definition 7.1 in \cite{GuirardelLevittUltimateJSJ}).
\end{definition}
Note that cylinders are subtrees. The (setwise) stabilizer of a cylinder is the commensurator of the stabilizer of any one of its edges $e$, i.e. the set of elements $g$ such that $g\Stab(e)g^{-1} \cap \Stab(e)$ has finite index both in $\Stab(e)$ and in $g\Stab(e)g^{-1}$. 

\begin{definition} The \textbf{boundary} of a cylinder $C$ is the set $\partial C$ of vertices of $C$ which either are of valence $1$ in $T$ or belong to at least two cylinders.
\end{definition}
Note that the stabilizer of a vertex $v$ in the boundary $\partial C$ of a cylinder $C$ is infinite cyclic if and only if $v$ has valence $1$ in $T$.

Given a $G$-tree $T$ with infinite cyclic edge stabilizers, we can construct the associated {\em tree of cylinders} $T_c$ (Definition 7.2 in \cite{GuirardelLevittUltimateJSJ} or Definition 4.3 in \cite{GuirardelLevittTreeOfCylinders}). It is obtained from $T$ as follows: the vertex set is the union 
$V_0(T_c) \cup V_1(T_c)$ where $V_0(T_c)$ contains a 
vertex $w'$ for each vertex $w$ of $T$ contained in the boundary of some cylinder, and $V_1(T_c)$ contains a vertex 
$v_c$ for each cylinder $c$ of $T$. There is an edge between vertices $w'$ and $v_c$ lying in $V_0(T_c)$ and $V_1(T_c)$ respectively 
if and only if $w$ belongs to the cylinder $c$.  

We get a tree which is bipartite: every edge in the tree of cylinders joins a vertex from $V_0(T_c)$ 
to a vertex of $V_1(T_c)$. Since the action of $G$ on $T$ sends cylinders to cylinders, the tree of cylinder admits an obvious $G$ action. The stabilizer of a vertex $w' \in V_0(T_c)$ is the stabilizer of the original vertex $w$, while the stabilizer of a vertex $v_C \in V_1(T_c)$ is the setwise stabilizer of the cylinder $C$. 
In a tree of cylinders $T_c$, cylinders are subtrees of diameter exactly $2$ (star graphs) whose center belongs to $V_1(T_c)$ and is not contained in any other cylinder.

For a general group $G$, the tree of cylinders need not have infinite cyclic edge stabilizers. But if $G$ is torsion-free hyperbolic for example (which will be our case in the sequel since we deal exclusively with free groups), then the commensurator of an infinite cyclic group is itself infinite cyclic, so the tree of cylinders does have infinite cyclic edge groups. In this case, by \cite[Lemma 7.3(5)]{GuirardelLevittUltimateJSJ}, $T_c$ is its own tree of cylinders, i.e. $(T_c)_c = T_c$.

\subsection{Normalized JSJ decompositions}

We recall the following from Section 4.1 of \cite{PerinSklinosForking}.

\paragraph{Deformation space.} The {\em deformation space} of a minimal cyclic $(G,A)$-tree $T$ is the set of all minimal cyclic $(G,A)$-trees $T'$ such that $T$ dominates 
$T'$ and $T'$ dominates $T$. The tree of cylinders is an invariant of the deformation space 
\cite[Lemma 7.3(3)]{GuirardelLevittUltimateJSJ}. 

A minimal cyclic $(G,A)$-tree is {\em universally elliptic} if its edge stabilizers are elliptic 
in every minimal cyclic $(G,A)$-tree. If $T$ is a universally elliptic cyclic $(G,A)$-tree, and $T'$ is any minimal cyclic $(G,A)$-tree, 
it is easy to see that there is a tree $\hat{T}$ which refines $T$ and dominates $T'$ (see \cite[Lemma 2.8]{GuirardelLevittUltimateJSJ}). 

\paragraph{JSJ trees.} A {\em cyclic JSJ tree for $G$ relative to $A$} is a minimal universally elliptic cyclic $(G,A)$-tree which dominates 
any other minimal universally elliptic cyclic $(G,A)$-tree. All these JSJ trees belong to a common deformation space, 
that we denote $JSJ_A(G)$. Guirardel and Levitt show that if $G$ is finitely presented and $A$ is finitely generated, 
the JSJ deformation space always exists (see \cite[Corollary 2.21]{GuirardelLevittUltimateJSJ}). It is easily seen to be unique.

\paragraph{Rigid and flexible vertices.} A vertex stabilizer in a (relative) JSJ tree is said to be {\em rigid} 
if it is elliptic in any cyclic $(G,A)$-tree, and {\em flexible} if not. In the case where $G$ is a torsion-free hyperbolic group and $A$ is a finitely 
generated subgroup of $G$ with respect to which $G$ is freely indecomposable, 
the flexible vertices of a cyclic JSJ tree of $G$ with respect to $A$ are 
{\em surface type} vertices \cite[Theorem 6.6]{GuirardelLevittUltimateJSJ}, i.e. their stabilizers are fundamental groups of hyperbolic surfaces with boundary, 
any adjacent edge group is contained in a maximal boundary subgroup, and any maximal boundary 
subgroup contains either exactly one adjacent edge group, or exactly one conjugate of $A$.

\paragraph{Minimal subtrees of JSJ trees and and JSJ of subgroups.} Assume $G$ is a finitely generated group and $A$ a set in $G$. We may consider the Grushko decomposition $G = G_A * G_1 * \ldots *G_k *\F_l$ for $G$ relative to $A$, where $G_A$ is the smallest free factor of $G$ containing $A$ and $\F_l$ is a free group of rank $l$.

\begin{lemma} \label{MinSubtreeJSJ} If $T \in JSJ_A(G)$, the minimal subtree $T^{min}_{G_A}$ of $G_A$ in $T$ is in the cyclic JSJ deformation space for $G_A$ relative to $A$, and the minimal subtree $T^{min}_{G_i}$ of each factor $G_i$ for $i=1, \ldots, k$ is in the cyclic JSJ deformation space for $G_i$.
\end{lemma}

\begin{proof} Suppose $T \in JSJ_A(G)$. Let $E_1(T)$ be the union of interior of trivially stabilized edges of $T$. Let $T_0, T_1, \ldots, T_r$ be representatives of the orbits of connected components of $T - E_1(T)$ where $v_A \in T_0$. Denote by $H_i$ the stabilizer of $T_i$. The group $G$ admits a decomposition as $H_0 * \ldots * H_r * \F_s$ for some $s$, and $A \leq G_0$.

Note that any refinement $T'_i$ of one of the trees $T_i$ can be extended to a refinement $T'$ of $T$, since the only edges of $T-T_i$ which are adjacent to $T_i$ are trivially stabilized. 

Show first that $T_i$ is universally elliptic among cyclic $H_i$-trees (respectively $(H_0,A)$-tree if $i=0$). If $\tilde{T}_i$ is another cyclic $H_i$-tree (respectively $(H_0,A)$-tree if $i=0$), we extend it to a cyclic $(G, A)$-tree $\tilde{T}$ by adding to the graph of group $\tilde{\Gamma}_i$ corresponding to $\tilde{T}_i$ a set of $k$ trivially stabilized edges with one endpoint in $\tilde{\Gamma}_i$ and the other endpoint with vertex group $H_j$ for $j\neq i$, and $s$ trivially stabilized loop attached to some vertex of $\Gamma_i$. Since $T$ is universally elliptic for $(G,A)$-trees, the stabilizer of an edge in $T_i$ fixes a vertex of $\tilde{T}$. Since all the edges of $\tilde{T}$ which are not in the orbit of $\tilde{T}_i$ are trivially stabilized, and $\Stab(e) \leq H_i = \Stab(\tilde{T}_i)$, it must in fact fix a vertex in $\tilde{T}_i$.

On the other hand, each $H_i$ is freely indecomposable: suppose by contradiction that $H_i$ acts minimally and non trivially on a tree $T_i^*$ with trivial edge stabilizers. Since $T_i$ is universally elliptic, there is a refinement $T'_i$ of $T_i$ which dominates $T^*_i$, obtained by replacing vertices in $T_i$ by the minimal subtree of their stabilizer in $T^*_i$ (in particular, we only add trivially stabilized edges). Note that $T'_i$ must have some trivially stabilized edges, in particular $T_i$ does not dominate $T'_i$ - there must be a vertex $v$ in $T_i$ whose stabilizer $\Stab(v)$ is not elliptic in $T'_i$. Now $T'_i$ can be extended to a proper refinement $T'$ of $T$. Note that to get $T'$ from $T$, we only added trivially stabilized edges, hence $T'$ is still universally elliptic. But $T$ cannot dominate $T'$ since $\Stab(v)$ is elliptic in $T$ and not in $T$ - this contradicts the maximality of $T$. 

Therefore the decomposition $G = H_0 * \ldots * H_r *\F_s$ of $G$ is in fact a Grushko decomposition for $G$, in particular $H_0 = G_A$ and $T^{min}_{G_A} \subseteq T_0$, $s=l$, $r=k$, and we may assume $H_i = G_i$ for $i=1, \ldots, k$. 

We now prove that the minimal subtree of $G_i=H_i$ in $T_i$ is a cyclic JSJ tree for $G_i$ (relative to $A$ if $i=0$). We already know it is universally elliptic since $T_i$ is. By Proposition 4.15 in \cite{GuirardelLevittUltimateJSJ}, it is enough to check that the JSJ of any vertex stabilizer in $T_i$ relative to the adjacent edge groups is trivial. Suppose by contradiction that some vertex $v \in T_i$ is such that $\Stab(v)$ admits a non trivial JSJ tree $T_v$ relative to the adjacent edge groups. Note that $\Stab(v)$ is not elliptic in $T_v$. By Proposition 4.15 of \cite{GuirardelLevittUltimateJSJ} refining $T$ at $v$ by $T_v$ gives a tree which should lie in the cyclic JSJ deformation space of $G$ relative to $A$ - a contradiction since in this refined tree, $\Stab(v)$ is no longer elliptic.
\end{proof}                                                                                                                                                                                                                                                                                                                        

\paragraph{Pointed tree of cylinders.} Assume $G$ is torsion-free hyperbolic and freely indecomposable with respect to $A$. Then, there is a preferred tree in the cyclic JSJ deformation space $JSJ_A(G)$ - indeed, in this case the tree of cylinders of $JSJ_A(G)$ itself lies in $JSJ_A(G)$ \cite[Theorem 9.18]{GuirardelLevittUltimateJSJ}. 

We slightly modify it to define the \textbf{pointed} cyclic JSJ tree for $G$ relative to $A$.
\begin{definition} Let $G$ be a torsion free hyperbolic group which is freely indecomposable relative to a non trivial finitely generated subgroup $A$. \textbf{The pointed cyclic JSJ tree of cylinders} for $G$ relative to $A$ is the tree $T$ defined as follows:
\begin{itemize}
\item[(i)] if $A$ is not infinite cyclic, $T$ is the tree of cylinders $T_c$ of $JSJ_A(G)$, and $v_A$ is the unique vertex of $T_c$ fixed by $A$;
\item[(ii)] if $A$ is cyclic, there is a unique edge $e=(v_a, p)$ adjacent to $v_A$, and removing its orbit from $T$ gives the tree of cylinder $T_c$ of $JSJ_A(G)$. Moreover, if $A$ stabilizes a cylinder $C$ in $T_c$, the vertex $p$ is the center of $C$.
\end{itemize}
\end{definition}

When we consider the minimal subtree of a subgroup $H$ containing $A$ in such a pointed tree, we require that $v_A$ belongs to $T^{min}_H$. 

\begin{definition} Let $T$ be a pointed cyclic JSJ tree for $G$ relative to $A$. Let $T'$ be the subtree of $T$ which lies in $JSJ_A(G)$.

A $v$ vertex of $T$ is called \textbf{rigid} if it is the base vertex $v_A$, or if it belongs to $T'$ and is rigid in $T'$. It is called \textbf{flexible} or \textbf{surface type} if it belongs to $T'$ and is flexible (surface type) in $T'$. 

It is said to be a \textbf{Z-type vertex} if it is distinct from the base vertex $v_A$ and its stabilizer is infinite cyclic.
\end{definition}

\begin{remark} Every vertex is either rigid or flexible (not both). All $Z$-type vertices, as well as the base vertex $v_A$, are rigid.
\end{remark}

\paragraph{Normalized JSJ decompositions for free groups.} We now restrict to the case where $G$ is a free group, but we drop the assumption that it is freely indecomposable relative to $A$. We define the normalized JSJ decompositions appearing in the main result:
\begin{definition} \label{NormalizedJSJDef}Let $\F$ be a finitely generated free group, let $A$ be a subset of $\F$ and let $\F_A$ be the smallest free factor of $\F$ containing $A$. A pointed cyclic JSJ tree $T$ relative to $A$ is said to be \textbf{normalized} if 
\begin{enumerate}
\item the minimal subtree $T^{min}_{\F_A}$ of $\F_A$ in $T$ is the pointed cyclic JSJ tree of cylinders of $\F_A$ relative to $A$,
\item trivially stabilized edges join a translate of the base vertex $v_A$ to a vertex which lies in a translate of $T^{min}_{\F_A}$.
\end{enumerate}
The decomposition $\Lambda$ associated to $T$ is called a normalized cyclic JSJ decomposition for $\F$ relative to $A$.
\end{definition}

\subsection{Statement of the main result}

We now give two definitions needed for the statement of the main result.
 
\begin{definition} \label{EnvelopesDef} Let $A$ be a subset of a free group $\F$ which is not contained in a proper free factor of $\F$. Let $T$ be the pointed cyclic JSJ tree of $\F$ relative to $A$, and let $T'$ be the subtree of $T$ which lies in $JSJ_A(\F)$ (i.e. the minimal subtree of $\F$ in $T$). 

Let $v$ be a rigid vertex of $T$. If $v$ is contained in $T'$, an \textbf{envelope of $v$ in $T$} is a union of edges adjacent to $v$ which meets at most one orbit of edge of each cylinder of $T'$.  If $v \not \in T'$ (so $v$ is a translate of $v_A$ and in particular has valence $1$), an envelope of $v$ consists at most of the unique edge adjacent to $v$.

Let $p: T \to \Lambda$ be the quotient map to the associated graph of groups. The image of an envelope of $v$ is a union of edges adjacent to $p(v)$, we call it an envelope of $p(v)$ in $\Lambda$.
\end{definition}

Note that the definition implies that an envelope of $p(v)$ is a star graph, i.e. no two edges in it have both endpoints in common.

\begin{figure}[ht!]
\centering
 \resizebox{\textwidth}{!}{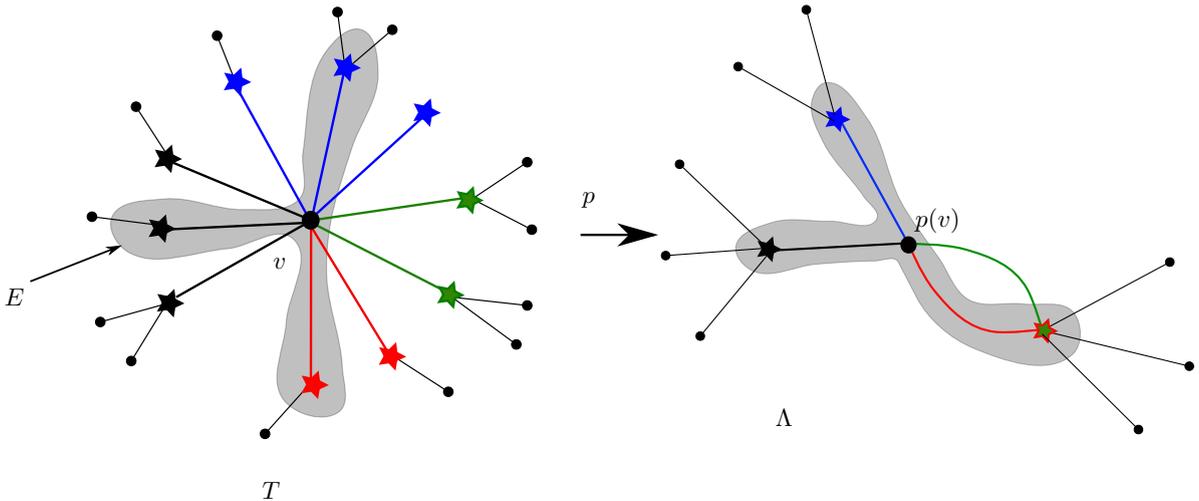}
\caption{$E$ is an envelope of $v$. Each color of edge adjacent to $v$ denotes a different orbit. Note that if we add a green edge to $E$, this is not an envelope anymore, since there exists a cylinder of $T$ which contains both red and green edges.}
\end{figure}

We generalize this to the case where $\F$ is not freely indecomposable with respect to $A$.
\begin{definition}Let $A$ be a subset of a free group $\F$, and denote by $\F_A$ the smallest free factor of $\F$ containing $A$. 
Let $T$ be a normalized pointed cyclic JSJ tree for $\F$ relative to $A$. 

Let $v$ be a rigid vertex of $T$. A subtree $E$ of $T$ is an envelope of $v$ in $T$ if $v \in E$ and there exists an element $g$ of $\F$ such that $g \cdot E$ is an envelope of $g \cdot v$ in $T^{min}_{\F_A}$. 
\end{definition}

Finally, we define the notion of blocks of a minimal subgraph of a subgroup $H$ which appears in the main result.
\begin{definition} \label{SandwichDef} Let $A$ be a subset of a free group $\F$, and denote by $\F_A$ the smallest free factor of $\F$ containing $A$.

Let $T$ be a normalized pointed cyclic JSJ tree for the free group $\F$ with respect to $A$. We say that an element $\beta$ of $\F$ is a \textbf{sandwich term} if the path $[v_A, \beta \cdot v_A]$ does not contain any translates of $v_A$ other than its endpoint, and can be subdivided into three subpaths $[v_A, u] \cup [u,v] \cup [v, g \cdot v_A]$ where $[v_A, u]$ and $[v, g \cdot v_A]$ are either empty or consist of a single trivially stabilized edge, and $[u,v]$ lies entirely in a translate of $T^{min}_{\F_A}$.

The image of the path $[u,v]$ in $\Lambda^{min}_{\F_A}$ is called the \textbf{imprint} of the sandwich term $\beta$ in $\Lambda^{min}_{\F_A}$ and denoted by $\imp(\beta)$.
The trivially stabilized edges crossed by $[v_A, \beta \cdot v_A]$ (by extension, their images in $\Lambda$) are called the trivially stabilized edges of the sandwich term $\beta$.
\end{definition}
	  
\begin{definition} \label{BlockDef} Let $A$ be a subset of a free group $\F$. Suppose $T$ is a normalized pointed cyclic JSJ tree for $\F$ relative to $A$ and denote by $\Lambda$ the graph of groups decomposition associated to $T$.

If $H$ is a subgroup $H$ of $\F$, we define ${\cal S}(H)$ to be the set of minimal subgraphs of sandwich terms $\beta$ such that some translate of $[v_A, \beta \cdot v_A]$ is contained in $T^{min}_H$. It is a collection of subgraphs of $\Lambda^{min}_H$.
   
We declare two subgraphs $\Lambda_i, \Lambda_j$ to be equivalent if their intersection contains an edge, or a vertex which is neither $Z$-type nor the base vertex $v_A$, and we consider the equivalence relation generated by this. A \textbf{block} of $\Lambda^{min}_H$  is the union of all the subgraphs of ${\cal S}(H)$ in a given equivalence class.
\end{definition}
Note that distinct blocks of $\Lambda^{min}_H$ are not necessarily disjoint, but they intersect at most in a disjoint union of $Z$-type vertices.

\begin{example} \label{ExBlocks}Suppose $\Lambda$ is a normalized JSJ decomposition of $\F$ relative to $A$ is as in Figure \ref{FigBlocks}, with rigid vertex groups $R, U, V, W$ with $A \leq R$, and a central $Z$-type vertex with vertex group $Z$, and a trivial edge with associated Bass-Serre element $t$. Let $u \in U, v\in V, w \in W$ be elements of the rigid vertex groups which do not lie in the adjacent edge group.

Suppose $c$ is the tuple $(u, t(vw)t^{-1})$. Then $\Lambda^{min}_{Ac}$ is the whole graph $\Lambda$. However, it consists of two blocks: one including the vertices with groups $R, U, Z$ and the edges between them, and the other consisting of the trivially stabilized edge, together with the edges between $V$ and $Z$ and the edge between $W$ and $Z$.
\end{example}

\begin{figure}[ht!]\label{FigBlocks}
\centering
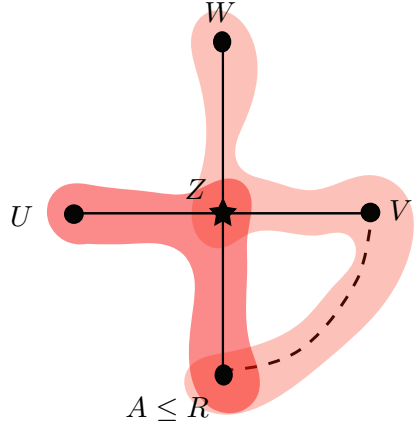
\caption{ The minimal subgraph of $(u, t(vw)t^{-1})$ has two blocks: the dark red one, which is the minimal subgraph of the sandwich term $u$, and the light red one, which is the minimal subgraph of the sandwich term $t(vw)t^{-1}$. 
(Recall that dot vertices represent rigid vertices, stars represent $Z$-type vertices, dashed edges are trivially stabilized while full ones have infinite cyclic stabilizers.)}
\end{figure}

We can now state the main result:
\begin{theorem} \label{MainResult}
	Let $A\subset \F$ be a set of parameters and $b,c$ be tuples from $\F$. Then $b$ is 
	independent from $c$ over $A$ if and only if there exists a normalized pointed cyclic $JSJ$ decomposition $\Lambda$ of $\F$ relative to $A$ in which the intersection of any two blocks of the minimal subgraphs $\Lambda^{min}_{Ab}, \Lambda^{min}_{Ac}$ is contained in a disjoint union of envelopes of rigid vertices. 
\end{theorem}

Let us give an example to illustrate the role of blocks in the result.
\begin{example} \label{ExBlocksIndep} Suppose $\Lambda$ and $c$ are as in Example \ref{ExBlocks}, and $b$ lies in $W$ but not in the adjacent edge group. Then the minimal subgraph of $Ab$ consists of the edges between the vertices corresponding to $R$ and $Z$, and to $Z$ and $W$. Thus $\Lambda^{min}_{Ab}$ intersects each one of the blocks of $\Lambda^{min}_{Ac}$ in an edge joining a rigid vertex to a $Z$-type vertex, i.e. in an envelope of a rigid vertex. Thus $b$ and $c$ are independent over $A$ (see a proof of this in Example \ref{ExBlocksProof}).

Note that the intersection of the minimal subgraphs $\Lambda^{min}_{Ab}$ and $\Lambda^{min}_{Ac}$ is not contained in a disjoint union of envelopes of rigid vertices, but the intersection of any two blocks of these subgraphs is.
\end{example}

\section{Proving independence} \label{IndependenceSec}

\subsection{A special case}

We fix a nonabelian finite rank free group $\F$.  Let $A\subset\F$ be a set of parameters and $b, c$ be 
tuples coming from $\F$. In this section we will prove one direction of Theorem \ref{MainResult} under the assumption that at least one of the tuples, 
say $b$, lives in the smallest free factor $\F_A$ of $\F$ containing $A$. 

In this case, we may assume that $\F$ is freely indecomposable relative to $Ac$, thus by Theorem 5.3 of \cite{PerinSklinosForking} the orbit $X:=Aut_{Ac}(\F).b$ 
is definable over $Ac$. But $X\subset\F_A$, thus by Fact \ref{ExtDef} the set $X$ is definable over $\F_A$.

The idea of the proof is to show that the orbit of $b$ under $\Aut_{Ac}(\F)$ contains the orbit of $b$ under $\Mod_A(\F)$, i.e. that any tuple that can be obtained from $b$ by a modular automorphism fixing $A$ can in fact be obtained by an automorphism fixing both $A$ and $c$. This proves that any set definable over $Ac$ and containing $b$ will contain an almost $A$-invariant subset, which is enough to prove independence.

We first show:
\begin{proposition} \label{ModEqual} Let $b, c$ be tuples in a free group $\F$, let $A \subseteq \F$, and denote by $\F_A$ the minimal free factor containing $A$. Assume that $\F$ is freely indecomposable with respect to $Ac$, and that $b \in \F_A$. 

Suppose that there exists a normalized pointed cyclic JSJ decomposition $\Lambda$ for $\F$ relative to $A$ in which the intersection of the minimal subgraph $\Lambda^{min}_{Ab}$ with any block of $\Lambda^{min}_{Ac}$ is contained in a disjoint union of envelopes of rigid vertices.

Then the orbit of $b$ under $\Aut_{Ac}(\F)$ contains the orbit of $b$ under $\Mod_{A}(\F_A)$.
\end{proposition}

Let us give the proof on an example first. 
\begin{example} \label{ExBlocksProof} We consider the example given in \ref{ExBlocks} and \ref{ExBlocksIndep}, where $\Lambda$ is as in Figure \ref{FigBlocks}, $c = (u, tvwt^{-1})$ and $b$ lies in $W$ and not in any adjacent edge group. Denote by $T$ the tree corresponding to $\Lambda$ and let $e_1, e_2, e_3, e_4$ be edges in $T$ adjacent to a common lift of the central $Z$-type vertex of $\Lambda$ which are lifts of the edges joining the vertices of $R$ and $Z$, $V$ and $Z$, $W$ and $Z$, and $U$ and $Z$ respectively. Denote by $\tau_i$ the Dehn twist associated to $e_i$ which restricts to a conjugation by the generator $z$ of $Z$ on $R$, $V$, $W$, and $U$ respectively and fixes pointwise all the other vertex groups.

Let $\theta \in \Mod_A(\F_A)$. By Lemmata 4.21 and 4.22 of \cite{PerinSklinosForking}, we can write
$$ \theta = \Conj(g) \circ \tau^{k_1}_1 \circ \tau^{k_2}_2 \circ \tau^{k_3}_3 \circ \tau^{k_4}_4$$
for $g = z^{-k_1}$ (since the composition fixes $R$ pointwise). Thus $\theta(b) = z^{k_3-k_1} b z^{k_1-k_3}$. We want to find an element $\theta' \in \Aut_{Ac}(\F)$ such that $\theta'(b) = \theta(b)$. 

We let $\theta'\mid_{\F_A} = \Conj(g) \circ \tau_1^{k_1} \circ \tau_2^{k_3} \circ \tau_3^{k_3} \circ  \tau^{k_1}_4$, and $\theta'(t) = tz^{k_1-k_3}$
Then clearly $\theta'(b) = z^{k_3-k_1} b z^{k_1-k_3}$. Now $\theta'(u) = \Conj(z^{-k_1})\tau_1^{k_1} (u)=u$, and $\theta'(vw) = \Conj(z^{-k_1})\tau_2^{k_3}(v) \circ \tau_3^{k_3}(w) = z^{k_3-k_1} vw z^{k_1-k_3}$ thus we get
$$\theta'(t(vw)t^{-1}) = \theta'(t) \theta'(vw) \theta'(t)^{-1} = tz^{k_1-k_3} z^{k_3-k_1}vwz^{k_1 -k_3} z^{k_3-k_1}t^{-1}$$  
so finally $\theta'(c) = c$.
\end{example}

\begin{proof} (of Proposition \ref{ModEqual})
Since $T$ is normalized, the minimal subtree $T^{min}_{\F_A}$ is the pointed cyclic JSJ tree of cylinders for $\F_A$ relative to $A$. 

By Lemmata 4.21 and 4.22 of \cite{PerinSklinosForking}, we can thus write any $\theta \in \Mod_{A}(\F_A)$ as 
$$ \theta = \Conj(z) \circ \rho_1 \circ \ldots \circ \rho_r$$
where each $\rho_i$ is an elementary automorphism, that is, a Dehn twist or a surface vertex automorphism supported on an edge or a surface vertex of the pointed JSJ tree $T^{min}_{\F_A}$. We want to show that there exists $\alpha \in \Aut_{Ac}(\F)$ such that $\theta(b) = \alpha(b)$. 

Lemma 4.21 tells us that up to changing $z$ we can permute the list of support of the $\rho_j$, and by Lemma 4.17 again up to changing $z$ we can replace $\rho_i$'s  by elementary automorphisms supported on a translate of $Supp(\rho_i)$. Note also that the composition of two elementary automorphism with the same support is again an elementary automorphism on this support.

Using all these facts, we may assume that $\rho_1, \ldots, \rho_k$ all have support in $T^{min}_{Ab}$, that $\rho_{k+1}, \ldots, \rho_r$ all have support outside any translate of $T^{min}_{Ab}$, and that different $\rho_i$'s have distinct support.

By Lemma 4.26 in \cite{PerinSklinosForking}, $\rho_{k+1} \circ \ldots \circ \rho_r$ is the composition of an element of $\Mod_{Ab}(\F_A)$ with a conjugation, hence the image of $b$ by $\theta$ is, up to conjugation, the same as the image of $b$ by $\rho_1 \circ \ldots \circ \rho_k$. Thus we may assume that $k=r$, in other words, that all the $\rho_i$ have support in $T^{min}_{Ab}$.

For each $i$, we define an automorphism $\rho'_i$ as follows: if $\Supp(\rho_i)$ does not lie in any translate of $T^{min}_{Ac}$, we set $\rho_i=\rho'_i$. If $\Supp(\rho_i)$ lies in a translate $g \cdot T^{min}_{Ac}$ of $T^{min}_{Ac}$, the hypothesis on the minimal subgraphs for $Ab$ and $Ac$ now ensures that this support is not a surface type vertex, but an edge $e_0$. Thus $\rho_i$ is a Dehn twist by some element $\gamma_i$ about $e_0$: we assume without loss of generality that it restricts to conjugation by $\gamma_i$ (and not to the identity) on the group associated to the non-Z type vertex of $e_0$. Denote by $e$ the image of $e_0$ in $\Lambda$, and by $\Gamma_e$ the block of $\Lambda^{min}_{Ac}$ containing it: again by our hypothesis on the minimal subgraphs, $e$ is the only edge of its cylinder which belongs both to $\Lambda^{min}_{Ab}$ and to $\Gamma_e$. We now let $\rho'_i$ be the product of Dehn twists by $\gamma_i$ about edges $e_0, e_1, \ldots, e_m$ which are representatives of orbits of edges of the cylinder of $e_0$ whose images lie in $\Gamma_e$ (we choose the Dehn twists to restrict to conjugation by $\gamma_i$ on the non $Z$-type vertex groups).

Note that up to conjugation, the image of $b$ by $\rho_1 \circ \ldots \circ \rho_r$ is the same as by $\theta'=\rho'_1 \circ \ldots \circ \rho'_r$, because we have only interspersed in the product some elementary automorphisms lying outside any translate of $T^{min}_{Ab}$.

Now we want to extend $\theta'$ to an automorphism $\alpha: \F \to \F$ fixing $Ac$. 

First, we choose an adequate presentation of $\F$: denote by $p: T \to T'$ the map defined by folding together in each cylinder all the edges whose image in $\Lambda$ belong to a common block of $\Lambda^{min}_{Ac}$. Denote by $\Lambda'$ the quotient graph of $T'$. Note that each $\rho'_i$ corresponds to an elementary automorphism of $T'$.

Pick a maximal subtree of $\Lambda'$ which does not include any trivially stabilized edges, lift it to a subtree $T^0$ of $T'$, and extend this to a Bass-Serre presentation $(T^0, T^1, \{t_{e}\}_{e \in E(T^1-T^0)})$ for $\Lambda'$ (recall Definition \ref{BassSerrePresentationDef}). Note that $\F = \F_A * \langle \{t_e\}_{e \in E_{triv^+}} \mid \; \rangle $ where $E_{triv^+}$ consists of all the trivially stabilized edges of $T^1$ of the form $(p(v_A), w)$ (to fix an orientation). Let $v^1, \ldots, v^s$ be the vertices of $T^0$ which come from rigid vertices of $T$ or whose inverse image by $p$ is not a single vertex, and let $H^1, \ldots, H^s$ be the stabilizers of $v^1, \ldots, v^s$. Note that each block of $(\Lambda')^{min}_{Ac}$ now consist of an envelope of the image of one of the vertices $v^i$ together with several trivially stabilized edges, and that each trivially stabilized edge is associated in this way with at most one vertex $v^i$.

Then $\theta'$ restricts to a conjugation by some element $\gamma_i$ on each $H^i$ and to the identity on the stabilizer of $p(v_A)$ in $T'$ since it is a product of elementary automorphisms associated to $T'$.

Define now $\hat{H}^i$ to be the subgroup generated by $H^i$ and all the $Z$-vertex groups adjacent to $H^i$ by a non trivially stabilized edge (so $\hat{H}^i$ is generated by $H^i$ together with some roots of elements of $H^i$). It is easy to see that $\theta'$ also restricts to a conjugation on $\hat{H}^i$.

Now each element in the tuple $c$ can be written as a product of the form
$$ g_1 (t_{e_1} h_1 t^{-1}_{f_1}) g_2 (t_{e_2} h_2 t^{-1}_{f_2}) \ldots g_m (t_{e_m} h_m t^{-1}_{f_m})g_{m+1} \; (*)$$
where for each $j$, 
\begin{enumerate}
\item $g_j \in \Stab(v_A)$ in $T'$;
\item there exists an index $l_j$ such that $e_j, f_j$ are trivially stabilized edges lying in the block of $(\Lambda')^{min}_{Ac}$ containing the image of $v^{l_j}$;
\item $h_j \in \hat{H}^{l_j}$.
\end{enumerate}

We now define $\alpha \in \Aut_{A}(\F)$ by setting $\alpha \mid_{\F_A} = \theta'$, and for each $i$ and each edge $e \in E_{triv^+}$ whose image lies in a block of $(\Lambda')^{min}_{Ac}$ containing the image of $v^i$, we set $\alpha(t_e) = t_e \gamma^{-1}_i$ (recall $\gamma_i$ is the element of $\F_A$ such that $\theta\mid_{H^i} = \Conj(\gamma_i)$). As we saw, $\F$ is the free product of $\F_A$ together with the free group generated by the elements $\{t_e\}_{e \in E_{triv^+}}$ so this indeed defines an automorphism of $\F$. 
But now it is easy to check using $(*)$ that $\alpha(c)=c$, hence $\alpha \in \Aut_{Ac}(\F)$.
\end{proof}

We can now prove:
\begin{proposition}\label{IndependenceSpecialCase}
Let $b, c$ be tuples in a free group $\F$, let $A \subseteq \F$, and denote by $\F_A$ the minimal free factor containing $A$. Assume that $\F$ is freely indecomposable with respect to $Ac$, and that $b \in \F_A$. 

Suppose that there exists a normalized   pointed cyclic JSJ decomposition $\Lambda$ for $\F$ relative to $A$ in which any two blocks of the minimal subgraphs $\Lambda^{min}_{Ab}$ and $\Lambda^{min}_{Ac}$ intersect in a disjoint union of envelopes of rigid vertices.

Then $b$ is independent from $c$ over $A$. 
\end{proposition}
\begin{proof}
We assume, for the sake of 
contradiction, that $b$ forks with $c$ over $A$. 

By Theorem 5.3 of \cite{PerinSklinosForking}, since $\F$ is freely indecomposable with respect to $Ac$, the orbit $X= \Aut_{Ac}(\F) \cdot b$ of $b$ under automorphisms of $\F$ fixing $Ac$ is definable over $Ac$. Since it is contained in any set definable over $Ac$ which contains $b$, we must have that $X$ forks over $A$. 

By Proposition \ref{ModEqual}, the set $X$ contains $\Mod_A(\F_A) \cdot b$. Now $\Mod_A(\F_A)$ has finite index in $\Aut_A(\F_A)$ (see \cite[Theorem 4.4]{RipsSelaHypI}) so $\Mod_A(\F_A) \cdot b$ is a non trivial almost $A$-invariant subset. By Lemma \ref{AtoFork}, we see that $X$ cannot fork over $A$, a contradiction.
\end{proof}

\subsection{Proof of independence in the general case}

We want to prove the first direction of Theorem \ref{MainResult}, namely
\begin{theorem} \label{DisjointnessImpliesIndependence} Let $b, c$ be tuples in a free group $\F$, and let $A \subseteq \F$. Denote by $\F_A$ the minimal free factor containing $A$.
Suppose there exists a normalized JSJ decomposition $\Lambda$ for $\F$ relative to $A$ in which the intersection of any two blocks of the minimal subgraphs $\Lambda^{min}_{Ab}$ and $\Lambda^{min}_{Ac}$ is contained in a disjoint union of envelopes of rigid vertices. Then $b \forkindep_A c$.
\end{theorem}

We first prove the result for $b,c$ sandwich terms (recall Definition \ref{SandwichDef}).
\begin{remark} Note that if $\beta$ is a sandwich term, the minimal subgraph $\Lambda^{min}_{A\beta}$ has only one block which is exactly $\Lambda^{min}_{A\beta}$.
\end{remark}

For sandwich terms, Theorem \ref{DisjointnessImpliesIndependence} holds.
\begin{proposition} \label{SandwichIndependence} If $\beta, \gamma$ are sandwich terms such that the intersection of the minimal subgraphs $\Lambda^{min}_{A\beta}$ and $\Lambda^{min}_{A\gamma}$ is contained in a disjoint union of envelopes of a rigid vertex, then $\beta$ and $\gamma$ are independent over $A$.
\end{proposition}

We will in fact prove a more general version of this proposition, which holds for tuples of sandwich terms satisfying certain conditions.

\begin{proposition} \label{JoinTheDots} Let $\F$ free group $\F$, and let $A \subseteq \F$. Denote by $\F_A$ the minimal free factor containing $A$. Let $T$ be a normalized JSJ tree for $\F$ with respect to $A$. Suppose $\bar{\beta}= (\beta^1, \ldots, \beta^q)$ and $\bar{\gamma}= (\gamma^1, \ldots, \gamma^r)$ are tuples of sandwich terms. 
	
Let $\Delta$ be a connected subgraph of groups of $\Lambda$ which contains $v_A$, and whose intersection with $\Lambda^{min}_{\F_A}$ is connected. Assume that
\begin{itemize}
\item the minimal subgraph of any element $\beta^j$ of $\bar{\beta}$ lies in $\Delta$;
\item the intersection of $\Delta$ with any block of $\Lambda^{min}_{A\bar{\gamma}}$ is contained in a disjoint union of envelopes of rigid vertices.
\end{itemize}
Then the tuples $\bar{\beta}$ and $\bar{\gamma}$ are independent over $A$.
\end{proposition}

\begin{proof} Pick a maximal subtree of $\Lambda_{\F_A}$ which extends a maximal subtree of $\Delta \cap \Lambda_{\F_A}$. Lift it to a subtree $T^0$ of $T$ which contains $v_A$, and extend this to a Bass-Serre presentation $(T^0, T^1, \{t_{e}\}_{e \in E(T^1-T^0)})$ for $\Lambda$ which contains a (connected) lift of $\Delta$, and in which the lifts of trivially stabilized edges are adjacent to $v_A$. 

Let $T_{\Delta \cap \Lambda_{\F_A}}$ be the connected component of the inverse image of $\Delta \cap \Lambda_{\F_A}$ in $T$ containing $v_A$, and let $H_0$ be the stabilizer of $T_{\Delta  \cap \Lambda_{\F_A}}$. By construction the minimal subgraph of $H_0$ is contained in $\Delta \cap \Lambda_{\F_A}$.

Now each element of the tuple $\bar{\beta} = (\beta^1, \ldots, \beta^q)$ can be written as $\beta^j = g_j t_{e_j} \beta^j_0 t^{-1}_{e'_j} g'_j$ where 
\begin{itemize}
\item $e_j, e'_j$ are trivially stabilized edges of $T^1 - T^0$ joining $v_A$ to a vertex in a translate of $T_{\Delta \cap \Lambda_{\F_A}}$,
\item $\beta^j_0 \in H_0$;
\item $g_j, g'_j \in \Stab(v_A)$.
\end{itemize}
Moreover, the edges $\{e_j, e_j' \mid j=1, \ldots, q\}$ do not lie in any translate of $T^{min}_{A \bar{\gamma}}$. In particular this means $\F$ admits a free product decomposition $\F = \F'* \F''$ where $\F_A \subseteq \F'$, $\bar{\gamma} \in \F'$ and $t_{e_j}, t_{e'_j} \in \F'' $ for all $j$. As a consequence, we get $t_{e_1} t_{e_1'} \ldots t_{e_q} t_{e_q'} \forkindep_{\emptyset} \F_A \bar{\gamma}$ by Theorem 1 of \cite{PerinSklinosForking}.

Let $\bar{\beta}_0 = (\beta^1_0, \ldots, \beta^q_0, g_1, g'_1, \ldots, g_q, g'_q)$. Since $\beta_0^j$ lies in $H_0$ for each $j$, and the minimal subgraph of $H_0$ lies in $\Delta$, we have that the minimal subgraph $\Lambda^{min}_{A \beta_0}$ lies in $\Delta$, hence its intersection with any block of $\Lambda^{min}_{A \bar{\gamma}}$ is contained in a disjoint union of envelopes of rigid vertice. 

Since $\bar{\beta}_0 \in \F_A$ we can apply Proposition \ref{IndependenceSpecialCase} to get that $\bar{\beta}_0 \forkindep_A \bar{\gamma}$. On the other hand, we have that $t_{e_1} t_{e'_1} \ldots t_{e_q} t_{e'_q} \forkindep_{\emptyset} \F_A \bar{\gamma}$ so $t_{e_1} t_{e'_1} \ldots t_{e_q} t_{e'_q}  \forkindep_{A \bar{\beta}_0} \bar{\gamma}$. Applying transitivity on $A \subseteq A\bar{\beta}_0 \subseteq A\bar{\beta}_0 t_{e_1} t_{e'_1} \ldots t_{e_q} t_{e'_q} $ and $\bar{\gamma}$, we get that $A\bar{\beta}_0 t_{e_1} t_{e'_1} \ldots t_{e_q} t_{e'_q} \forkindep_A \bar{\gamma}$, which implies $\bar{\beta} \forkindep_A \bar{\gamma}$ as wanted. 
\end{proof}

\begin{remark} Proposition \ref{SandwichIndependence} follows from Proposition \ref{JoinTheDots}. Indeed, let $P$ be a (possibly trivial) path in $\Lambda^{min}_{\F_A}$ from $v_A$ to $\imp(\beta) \cup \imp(\gamma)$ (recall that the imprint of a sandwich term $\beta$ is the image of the middle segment of $[v_A, \beta \cdot v_A]$ which lies in a translate of $T^{min}_{\F_A}$). Without loss of generality the endpoint of $P$ lies in $\imp(\beta)$ - if we let $\Delta = P \cup \Lambda^{min}_{\beta}$, the hypotheses of Proposition \ref{JoinTheDots} are satisfied and as a conclusion we get that $\beta \forkindep_A \gamma$.
\end{remark}
  
Now we will prove that if $b,c$ satisfy the hypotheses of Theorem \ref{DisjointnessImpliesIndependence}, their elements can be written as a product of sandwich terms whose minimal subgraphs form "almost disjoint" blocks.

\begin{proposition} \label{ProductOfSandwiches} Let $\F$ be a free group, let $A \subseteq \F$ and let $b, c$ be tuples in $\F$. 

Suppose there exists a normalized cyclic pointed JSJ decomposition $\Lambda$ for $\F$ relative to $A$ in which the intersection of any two blocks of the minimal subgraphs $\Lambda^{min}_{Ab}$ and $\Lambda^{min}_{Ac}$ lies in a disjoint union of envelopes of rigid vertices.

Then there exists sets $B$ and $C$ of sandwich terms such that 
\begin{enumerate}
\item  the intersection of a block of $\Lambda^{min}_{AB}$ with a block of $\Lambda^{min}_{AC}$
is a disjoint union of envelopes of rigid vertices;
\item each element $b^i$ of $b$ (respectively $c^j$ of $c$) can be written as a product of elements of $B$ (respectively of $C$);
\end{enumerate}
\end{proposition}

\begin{proof} Let $b^j$ be an element of the tuple $b$. Consider the path $[v_A, b^j \cdot v_A]$ - it lies in $T^{min}_{Ab}$. We subdivide it into finitely many subpaths with endpoints the translates of $v_A$ that appear in $[v_A, b^j \cdot v_A]$. Each of these subpaths is thus of the form $[g \cdot v_A, h\cdot v_A]$ and contains at most two trivially stabilized edges which must appear as the first or the last edge of the path (because any trivially stabilized edge in $T$ is adjacent to a translate of $v_A$). This means precisely that $\beta = g^{-1}h$ is a sandwich term, and moreover $[v_A, \beta \cdot v_A]$ lies in a translate of $T^{min}_{Ab}$, thus $\Lambda^{min}_{\beta}$ lies in a block of $\Lambda^{min}_{Ab}$. By construction, $b^j$ is a product of such terms.
	
Similarly, each term in the tuple $c$ can be written as a product of sandwich terms $\gamma$ such that $\Lambda^{min}_{\gamma}$ lies in a block of $\Lambda^{min}_{Ac}$. This proves the claim.
\end{proof}

To prove Theorem \ref{DisjointnessImpliesIndependence}, the idea is to divide up the sets of sandwich terms $B$ and $C$ obtained above so that their minimal subgraphs are alternatively contained in a growing chain of connected subgraphs of $\Lambda$, and then to use Proposition \ref{JoinTheDots} inductively to prove the result. Let us see how this works on an example.

\begin{example} \label{ForkingCalculusEx} Suppose $\Lambda$ is as in Figure \ref{FigForkingCalculusEx}. Suppose $b = (t_4u_4t^{-1}_4)(t_3u_3u_2t^{-1}_2)$ and $c = (s_2v_2s^{-1}_2)(s_3v_3v_4v_3 s^{-1}_3)$ where $u_i, v_i$ are elements of the vertex group $U_i$, and $t_i$ (respectively $s_i$) is a Bass-Serre element corresponding to the trivially stabilized edge on the left-hand side (respectively right-hand side) joining the base vertex to the vertex stabilized by $U_i$. Intersections of blocks of the minimal subgraphs of $Ab, Ac$ are disjoint union of vertices stabilized by the $U_i$'s.

\begin{figure}[ht!]\label{FigForkingCalculusEx} 
\centering
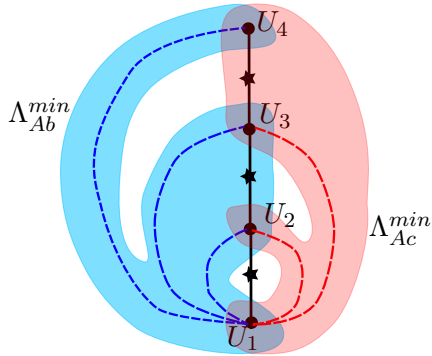
\caption{The JSJ decomposition in example \ref{ForkingCalculusEx}. (Recall that dot vertices represent rigid vertices, stars represent $Z$-type vertices, dashed edges are trivially stabilized while full ones have infinite cyclic stabilizers.) }
\end{figure}

Decompose $b$ and $c$ as products of sandwich terms, as $b = \beta_1\beta_2$ and $\gamma=\gamma_1\gamma_2$ where $\beta_1=t_4u_4t^{-1}_4, \beta_2=t_3u_3u_2t^{-1}_2$ and $\gamma_1 =  s_2v_2s^{-1}_2, \gamma_2 = s_3v_3v_4v_3 s^{-1}_3$. 

Applying Proposition \ref{JoinTheDots} to the subgraph $\Delta$ which consists of the vertices stabilized by $U_1, U_2$ the path in $\Lambda_A$ between them, and the trivially stabilized edge associated to $s_2$, we get that $\beta_2 \forkindep_A \gamma_1$. Applying it in the case where $\Delta$ is the path in $\Lambda_A$ joining the vertices stabilized by $U_3$ and $U_1$, together with the edges labelled by $s_2, t_2$ and $t_3$, we get that $\gamma_2 \forkindep_A \gamma_1\beta_2$ so in particular $\gamma_2 \forkindep_{A \gamma_1} \beta_2$. Thus by transitivity (with $A \subseteq A\gamma_1 \subseteq A\gamma_1\gamma_2$) we get that $\beta_2 \forkindep_A A\gamma_1\gamma_2$.

We can then apply again Proposition \ref{JoinTheDots} (taking $\Delta$ to be complement in $\Lambda$ of the edge labeled by $t_4$) to see that $\beta_1 \forkindep_A \gamma_1\gamma_2\beta_2$ so $\beta_1 \forkindep_{A\beta_2} \gamma_1\gamma_2$. So applying transitivity with $A \subseteq A \beta_2 \subseteq A\beta_1 \beta_2$ and $\gamma_1 \gamma_1$ we get $\beta_1\beta_2 \forkindep_A \gamma_1\gamma_2$, which implies $b \forkindep_A c$.
\end{example}

The next proposition generalizes this example to prove what can be obtained from an inductive application of \ref{JoinTheDots}. 
\begin{proposition} \label{ForkingCalculus} Suppose $\Delta_0=\{v_A\} \subseteq \Delta_1 \subseteq \ldots \subseteq \Delta_s=\Lambda$ is a chain of connected subgraphs of groups of $\Lambda$. Let $B,C$ be sets of sandwich terms such that no trivially stabilized edge is associated both to an element $\beta \in B$ and an element $\gamma \in C$.

For any $\beta_0 \in B$ (respectively $\gamma_0 \in C$), denote by $\Block_B(\beta_0)$ (respectively $\Block_C(\gamma_0)$) the block of $\Lambda^{min}_{AB}$ (respectively $\Lambda^{min}_{AC}$) containing $\Lambda^{min}_{\beta_0}$ (respectively $\Lambda^{min}_{\gamma_0}$).
	
Suppose $B, C$ can be partitioned as  $B = \bigsqcup^{\lfloor (s-1)/2 \rfloor}_{i=0} B_i$ and $C = \bigsqcup ^{\lfloor s/2 \rfloor}_{i=0} C_i$ in such a way that
\begin{itemize}
\item any two terms $\beta, \beta' \in B$ (respectively $\gamma,\gamma'$) such that $\Block_B(\beta) = \Block_B(\beta')$ (respectively $\Block_C(\gamma) = \Block_C(\gamma')$) are in the same subset $B_i$ (respectively $C_i$);
\item  for each $i$, and each $\beta \in B_i$, $\Block_B(\beta)$ lies in $\Delta_{2i+1}$, and its intersection with $\Delta_{2i}$ is a disjoint union of envelopes of rigid vertices;
\item for each $i$, and each $\gamma \in C_i$, $\Block_C(\gamma)$ lies in $\Delta_{2i}$, and its intersection with $\Delta_{2i-1}$ is a disjoint union of envelopes of rigid vertices.
\end{itemize}
Then $B \forkindep_A C$.
\end{proposition}

\begin{proof} We prove this by induction on the length of the chain of subgraphs of groups. If $s=1$, we have that the minimal subgraph of any term in $C$ is contained in $\{v_A\}$, hence applying Proposition \ref{JoinTheDots} to $\Delta = \{v_A\}$ gives the result.

Suppose it holds for chains of length at most $s$, and let us prove it for a chain $\Delta_0=\{v_A\} \subseteq \Delta_1 \subseteq \ldots \subseteq \Delta_s \subseteq \Delta_{s+1}=\Lambda_{\F_A}$.

We prove this for $s = 2k$, the odd case is very similar. The hypotheses give us partitions $B = \bigsqcup^{k}_{i=0} B_i$ and $C = \bigsqcup^{k}_{i=0} C_i$. We have
\begin{enumerate}
\item $B_0 \ldots B_{k-1}B_{k} \forkindep_{A B_0\ldots B_{k-1}} C_0\ldots C_{k}$. Indeed, the subgraph of groups $\Delta_{s}$ satisfies that the minimal subgraphs of all the terms in $B_0\ldots B_{k-1}C_0\ldots C_k$ lie in $\Delta_{s}$, and any block of $\Lambda^{min}_{AB_k}$ intersects $\Delta_{s}$ in a disjoint union of envelopes of rigid vertices, so the hypotheses of Proposition \ref{JoinTheDots} are satisfied: we get that $B_0\ldots B_{k-1}C_0\ldots C_{k-1} C_k \forkindep_A B_{k}$ so by transitivity $C_0\ldots C_k \forkindep_{AB_0\ldots B_{k-1}} B_{k}$.
\item $B_0 \ldots B_{k-1} \forkindep_{A} C_0 \ldots C_{k}$. This follows from applying the induction hypothesis to the chain $\Delta_0 \subseteq \Delta_1 \subseteq \ldots \subseteq \Delta_{s-1} \subseteq \Delta_{s+1}= \Lambda_{\F_A}$ (we skip $\Delta_{s}$) and to the sets $\bigsqcup^{k-1}_{i=0} B_i$ and $C$, we get that $B_0 \ldots B_{k-1} \forkindep_A C_0 \ldots C_{k}$. 
\end{enumerate}
From 1. and 2. we get by transitivity that $B_0 \ldots B_{k} \forkindep_A C_0 \ldots C_{k}$. 
\end{proof}

We can finally prove Theorem \ref{DisjointnessImpliesIndependence}.
\begin{proof}Let $B$ and $C$ be the sets of sandwich terms given by Proposition \ref{ProductOfSandwiches}. 
	
We build a sequence of connected subgraphs  $\Delta_0 =\{v_A\} \subseteq \Delta_1 \subseteq \ldots \subseteq \Delta_r = \Lambda$ satisfying the conditions of Proposition \ref{ForkingCalculus} as follows. 

Suppose we have built $\Delta_i$: if $\Delta_i$ contains all the blocks $\Block_B(\beta)$ and $\Block_{C}(\gamma)$ for all $\beta \in B$ and $\gamma \in C$, we set $\Delta_{i+1} = \Lambda_{\F_A}$ and we are done. 

If not, consider all the blocks $\Block_B(\beta)$ which are not contained in $\Delta_i$ and such that there is a path in $\Lambda_{\F_A}$ between ${\Delta_i}$ and $\Block_B(\beta)$ which does not intersect any $\Block_C(\gamma)$ in more than a disjoint union of envelopes of rigid vertices. If such components exist, we build $\Delta_{i+1}$ by adding them to $\Delta_i$ together with the appropriate paths joining them to $\Delta_i$, and we set $B_{\lfloor i/2 \rfloor}$ to be the set of all terms of $B$ whose minimal subgraphs lies in those components we have added. If not, we consider the blocks $\Block_C(\gamma)$ which are not contained in $\Delta_i$ and for which there are paths in $\Lambda_{\F_A}$ between  ${\Delta_i}$ and $\Block_C(\gamma)$ which do not intersect any blocks $\Block_B(\beta)$ in more than a disjoint union of envelopes rigid vertices (there must be some), and we proceed similarly. 

By construction the chain $\Delta_0 =\{v_A\} \subseteq \Delta_1 \subseteq \ldots \subseteq \Delta_r = \Lambda$ and the partitions given for the sets $B$ and $C$ satisfy the hypotheses of Proposition \ref{ForkingCalculus} (up to switching between $B$ and $C$), hence $B \forkindep_A C$.
\end{proof}

\section{Proving forking} \label{ForkingSec}

In this subsection we prove the left to right direction of Theorem \ref{MainResult}, namely: 
\begin{proposition}\label{LefttoRight}
Let $A\subset \F$ be a set of parameters and $b,c$ be tuples from $\F$. Suppose that in any 
normalized cyclic $JSJ$ decomposition $\Lambda$ for $\F$ relative to $A$, there are blocks of the minimal subgraphs $\Lambda^{min}_{Ab}$ and $\Lambda^{min}_{Ac}$ whose intersection is not contained in a disjoint union of envelopes of rigid vertices. 

Then $b$ forks with $c$ over $A$.
\end{proposition}

The following proposition shows that some elements of the algebraic closure of $Ab$ can be read off in the JSJ decomposition of the minimal free factor containing $A$.

First, we show the following fact about algebraic closures in the free group. 
\begin{proposition} Let $\F$ be a finitely generated free group, let $H \leq \F$ be a subgroup relative to which $\F$ is freely indecomposable. Let $\Lambda$ be any JSJ decomposition of $\F$ relative to $H$. Then the vertex subgroup containing $H$ is contained in $\acl(H)$, and if $v$ is a rigid vertex of $\Lambda$, the conjugacy class of any tuple of elements of a subgroup associated to $v$ lies in $\acl^{eq}(H)$.
\end{proposition}

The first part of this result was first proved in \cite{OuldHoucineVallinoAlgClosure}.
\begin{proof} By Theorem 5.3 of \cite{PerinSklinosForking}, the orbit $X:=Aut_{H}(\F).g$ of a tuple $g$ under automorphisms fixing $H$ is definable over $H$. In fact, it is contained in any set definable over $H$ containing $H$. 

Thus any tuple whose orbit under $\Aut_H(\F)$ is finite lies in $\acl(H)$, and the conjugacy class of any tuple whose orbit intersects finitely many conjugacy classes lies in $\acl^{eq}(H)$. Now a tuple lying in the vertex subgroup containing $H$ is fixed by elements of $\Mod_H(\F)$, and elements of $\Mod_H(\F)$ restrict to a conjugation on rigid vertex groups of any JSJ decomposition. Since $\Mod_H(\F)$ has finite index in $\Aut_H(\F)$, this proves the result.
\end{proof}
 
\begin{proposition} \label{SeeAclEqInLambda} Let $\F$ be a free group, and $A \subseteq \F$ a set of parameters. Let $b$ be a tuple in $\F$, such that $\F$ is freely indecomposable with respect to $Ab$. Denote by $\F_A$ the minimal free factor of $\F$ containing $A$, by $\Lambda_{\F_A}$ the cyclic pointed JSJ decomposition of $\F_A$ with respect to $A$, and by $T$ the corresponding tree.
	
There are subgraphs $\Gamma^1_{Ab}, \ldots, \Gamma^r_{Ab}$ of $\Lambda_{\F_A}$ such that the following hold:
\begin{enumerate}
    \item if $i \neq j$, then $\Gamma^i_{Ab}$ and $\Gamma^j_{Ab}$ intersect at most in a disjoint union of $Z$-type vertex;
	\item Suppose $e=(z,x)$ and $e'=(z,y)$ are adjacent edges of $\Gamma^i_{Ab}$ for some $i$, that $x,y$ are neither $Z$-type nor surface type vertices, and let $\hat{e}=(\hat{x}, \hat{z}), \hat{e}'=(\hat{z}, \hat{y})$ be adjacent lifts of $e, e'$ in $T$. Denote by $s_x, s_y$ some generating tuples for the stabilizers of $\hat{x},\hat{y}$. Then the conjugacy class of $(s_x, s_y)$ is in $\acl^{eq}(Ab)$.
	\item Suppose $v$ is a surface type vertex of $\Gamma^i_{Ab}$ for some $i$, and let $\hat{v}$ be a lift of $v$ in $T$. Then $\acl^{eq}(Ab)$ contains the conjugacy class of an element $g$ in the stabilizer of $\hat{v}$ which corresponds to a non boundary parallel simple closed curve.
\end{enumerate}
Moreover, there is a normalized cyclic JSJ decomposition $\Lambda$ of $\F$ relative to $A$ such that the intersection of any block of the minimal subgraph $\Lambda^{min}_{Ab}$ with $\Lambda^{min}_{\F_A}$ (which is isomorphic to $\Lambda_{\F_A}$) is contained in $\Gamma^i_{Ab}$ for some $i$.
\end{proposition}

\begin{proof}Consider the pointed cyclic JSJ decomposition $\Lambda(1)$ for $\F$ relative to $Ab$ corresponding to the pointed cyclic tree of cylinders $T(1)$. 

Denote by $\Lambda(2)$ the graph of groups obtained from $\Lambda(1)$ by collapsing all the edges whose corresponding groups are not universally elliptic in $(\F, A)$-trees. Denote by $v_0$ the vertex of $\Lambda(2)$ stabilized by $Ab$, and by $v_1, \ldots, v_s$ the (other) vertices of $\Lambda(2)$ which are not of surface type. 

By Proposition 4.15 in \cite{GuirardelLevittUltimateJSJ}, refining $\Lambda(2)$ at each vertex $v_i$ by a normalized cyclic JSJ decomposition $\Gamma_i$ of the corresponding vertex group relative to its adjacent edges groups (a pointed normalized JSJ decomposition relative to $A$ if $i=0$) gives a JSJ decomposition for $\F$ relative to $A$.

Note also that if $e$ is an edge of $\Lambda(2)$ joining a $Z$-type vertex $z$ to one of the vertices $v_i$, and whose stabilizer preserves a cylinder in $\Gamma_i$, in building this decomposition we attach it to the center of this cylinder. The cylinder of the edge $e$ in the graph thus obtained will consist of a refinement of its cylinder in $\Lambda(2)$ by a star-graph cylinder at several of its boundary points. In particular, note that the JSJ decomposition thus obtained might not be normalized.

Now if $A$ stabilized an edge in $\Lambda(2)$ but not in $\Gamma_0$, necessarily it stabilizes an edge $e$ which we attached to the vertex $v_A$ of $\Gamma_0$. We further modify our JSJ decomposition by adding a valence $1$ edge $e'$ to the cylinder corresponding to $e$, setting the basepoint $v_A$ to be the valence $1$ vertex thus created, and displacing the endpoints of the trivially stabilized edges from the old basepoint to the new. We also add to $\Gamma_0$ the edges $e,e'$.

We denote by $\Lambda(3)$ the graph of groups thus obtained: it is a pointed cyclic JSJ decomposition for $\F$ relative to $A$. By Lemma \ref{MinSubtreeJSJ}, the minimal subtree of $\F_A$ in the tree corresponding to $\Lambda(3)$ lies in $JSJ_{A}(\F_A)$.

Denote by $\Gamma^1_{Ab}, \ldots, \Gamma^r_{Ab}$ the subgraphs of $\Lambda(3)$ corresponding to the connected components of the complement of trivially stabilized edges in $\Gamma_0$.

\begin{figure*}[ht!]\label{Fig4Graphs}
\centering
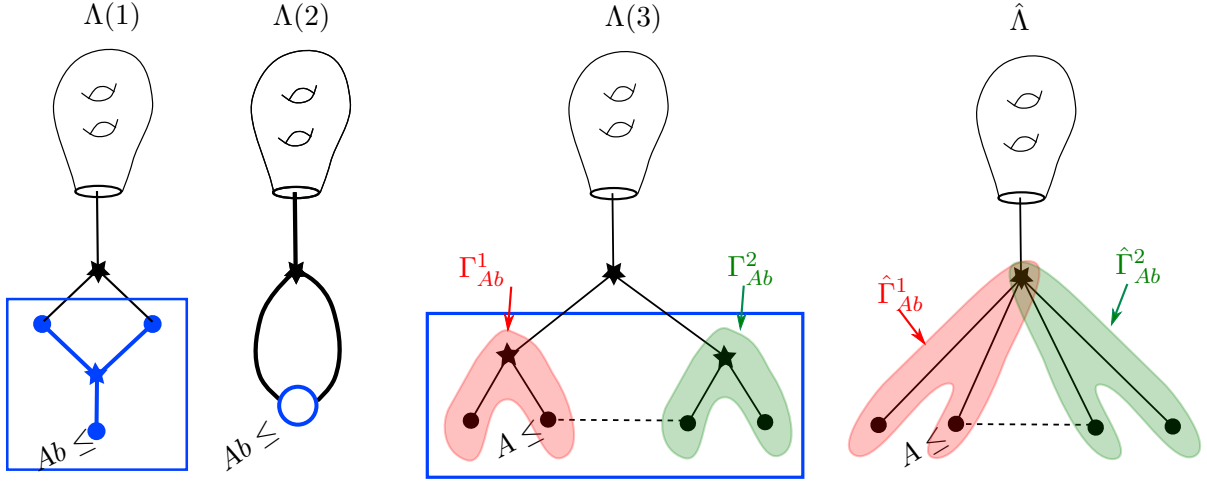
\caption{An example of the construction of the various graphs appearing in the proof of Proposition \ref{SeeAclEqInLambda}. The edges of $\Lambda(1)$ which are not universally elliptic in $(\F_A, A)$-trees are those enclosed by the blue box. As usual star vertices are $Z$-type, dot vertices are rigid, dashed edges are trivially stabilized while full edges correspond to infinite cyclic edge groups.}
\end{figure*}

Recall that there were edges of $\Lambda(2)$ which were attached in $\Gamma_0$ to center of cylinders in the construction of $\Lambda(3)$. By collapsing such edges, we do not change the deformation space, so we get a pointed cyclic JSJ decomposition $\hat{\Lambda}$ for $\F$ relative to $A$ in which cylinders of the minimal subgraph of $\F_A$ are star graphs. 
Moreover, trivially stabilized edges come from the pointed normalized JSJ decomposition $\Gamma_0$, thus they are attached to the base vertex $v_A$. Thus $\hat{\Lambda}$ is a normalized JSJ decomposition. 

Denote by $\hat{\Gamma}^1_{Ab}, \ldots, \hat{\Gamma}^r_{Ab}$ the images of $\Gamma^1_{Ab}, \ldots, \Gamma^r_{Ab}$ in $\hat{\Lambda}$. The minimal subgraph of $Ab$ in $\Lambda(3)$ is contained in $\Gamma_0$. Therefore the minimal subgraph $\hat{\Lambda}^{min}_{Ab}$ of $Ab$ in $\hat{\Lambda}$ is contained in the image of $\Gamma_0$, and the intersection of each block of $\hat{\Lambda}^{min}_{Ab}$ with the minimal subgraph of $\F_A$ is contained in one of the components $\hat{\Gamma}^i_{Ab}$.

Let us see that the other conclusions of the proposition hold for $\hat{\Gamma}^1_{Ab}, \ldots, \hat{\Gamma}^r_{Ab}$. Item 1. is immediate by construction of $\hat{\Lambda}$.

Suppose $e,e'$ are adjacent edges of $\hat{\Gamma}^i_{Ab}$ for some $i$, that $x,y$ are neither $Z$-type nor surface type vertices, and let $\hat{e}=(\hat{x}, \hat{z}), \hat{e}'=(\hat{z}, \hat{y})$ be adjacent lifts of $e, e'$ in $T$. Denote by $s_x, s_y$ some generating tuples for the stabilizers of $\hat{x},\hat{y}$. Note that the stabilizers of $\hat{x}$, $\hat{y}$ are elliptic in any $\F_A$-tree in which $A$ is elliptic, hence they each stabilize a non $Z$-, non surface type vertex of $T(1)$. Since they have a non trivial intersection, they stabilize either a common vertex, or distinct vertices which belong to a common cylinder. In the first case, we have that the conjugacy class of $(s_x,s_y)$ lies in $\acl^{eq}(Ab)$ so we are done. In the second case, the stabilizer of the edges $f, f'$ joining these two distinct vertices is commensurable to the stabilizer of $e$, hence is universally elliptic among $(\F, A)$-trees. Thus $f, f'$ are not collapsed in $T_{\Lambda(2)}$ and $s_x, s_y$ stabilize distinct vertices in $T_{\Lambda(2)}$. But this contradicts the fact that $x,y$ lie in the same $\hat{\Gamma}^i_{Ab}$ and that we took adjacent lifts of $e,e'$.

Similarly, let $v$ be a surface vertex of $\hat{\Gamma}^i_{Ab}$. The boundary subgroups of the stabilizer $S_v$ of $\hat{v}$ are elliptic in any  $\F_A$-tree in which $A$ is elliptic, hence they are elliptic in $T(1)$, and stabilize non surface type vertices. Thus there is a system ${\cal C}$ of non boundary parallel disjoint simple closed curves on the surface associated to $v$ such that the subgroups of $S_v$ corresponding to the connected components of the complement of ${\cal C}$ are elliptic in $T(1)$. If ${\cal C}$ is non empty, any one of the elements of $S_v$ corresponding to a curve in ${\cal C}$ corresponds to an edge group of $T(1)$, hence its conjugacy class lies in $\acl^{eq}(Ab)$. If ${\cal C}$ is empty, this implies $S_v$ is elliptic in $T(1)$. If it stabilizes a non surface type vertex the conjugacy class of any element of $S_v$ lies in $\acl^{eq}(Ab)$ and we are done. 

If it stabilizes a surface type vertex $w$, the boundary subgroups of $S_v$ are boundary subgroups of $S_w$ hence $S_v$ has finite index in $S_w$. But this implies $S_w$ is contained in $\F_A$, and moreover all its boundary subgroups are universally elliptic in $(\F_A, A)$-trees since they have finite index subgroups which are. But then this means that none of the edge adjacent to $w$ are collapsed in $T(2)$, so the image of $w$ is of surface type in $T(2)$ and thus its image in $\hat{T}$ does not lie in any of the $\hat{\Gamma}^i_{Ab}$. Since this image is stabilized by $S_v$ which does stabilize a vertex in some $\hat{\Gamma}^i_{Ab}$, we get a contradiction.
\end{proof}

The following proposition will help us conclude
\begin{proposition} \label{NotInAcl}
Let $\F$ be a free group, and $A \subseteq \F$ a set of parameters. Let $T$ be a normalized JSJ tree for $\F$ relative to $A$.
If $e=(z,x)$ and $e'=(z,y)$ are two non trivially stabilized edges of $T$ with $x,y$ non $Z$-type vertices, the conjugacy class of $(s_x, s_y)$ is not in $\acl^{eq}(A)$.

If $v$ is a surface type vertex of $T$, the conjugacy classes of two elements $g,h$ corresponding to non boundary parallel simple closed curves on the surface associated to $v$ fork over $A$.

\end{proposition}

\begin{proof} Up to conjugating $(s_x, s_y)$, we can assume that it lies in $\F_A$. It is thus enough to show that the orbit of $(s_x, s_y)$  under $\Aut_A(\F_A)$ contains infinitely many distinct conjugacy classes. 

One of $x,y$ is not the basepoint so without loss of generality the stabilizer of $x$ is not cylic. Let $\tau_e$ be a Dehn twist about $e$ by some element $\epsilon$ of $\Stab(e)$. For any $m$, the equation 
$$ \gamma (s_x, s_y) \gamma^{-1} = (s_x, \epsilon^m s_y \epsilon^{-m})$$
implies that $\gamma=1$ since $\langle s_x \rangle$ is not abelian. Thus the pairs $\tau^m_e(s_x,s_y)$ lie in distinct conjugacy classes.

For the second part, assume without loss of generality that $g,h $ lie in $\F_A$. By Theorem 2 of \cite{PerinSklinosForking}, $g$ forks with $h$ over $A$. This implies the result.
\end{proof}

We can now prove Proposition \ref{LefttoRight}.
\begin{proof} Assume $b \forkindep_A c$. We will construct a normalized JSJ decomposition for $\F$ relative to $A$ in which any two blocks of the minimal subgraphs $\Lambda^{min}_{Ab}$ and $\Lambda^{min}_{Ac}$ intersect at most in a disjoint union of envelopes of rigid vertices.

Denote by $\F_A, \F_{Ab}$ and $\F_{Ac}$ the smallest free factors of $\F$ that contains $A, Ab$ and $Ac$ respectively. Let $\Lambda_A$ be the cyclic pointed JSJ decomposition of $\F_A$ relative to $A$, and let $H$ be such that $\F=H*\F_A$. Let $(\Gamma^j_{Ab})^r_{j=1}$ and $(\Gamma^j_{Ac})^s_{j=1}$ be the (possibly disconnected) subgraphs of $\Lambda_A$ obtained by applying Proposition \ref{SeeAclEqInLambda} to $Ab, Ac$ respectively. We also get a normalized JSJ decomposition $\tilde{\Lambda}_{Ab}$ (resp. $\tilde{\Lambda}_{Ac}$) for $\F_{Ab}$ (resp. $\F_{Ac}$), in which each block of the minimal subgraph for $Ab$ (resp. $Ac$) lies inside one of the subgraphs $\Gamma^j_{Ab}$ (resp. $\Gamma^j_{Ac}$). We can extend these JSJ decompositions to JSJ decompositions $\tilde{\Lambda}^+_{Ab}$ and $\tilde{\Lambda}^+_{Ac}$ of $\F$ relative to $A$ by adding a number of trivially stabilized loops to the base vertex.

Suppose there exist $i,j$ such that the intersection of $\Gamma^i_{Ab}$ and $\Gamma^j_{Ac}$ is not a disjoint union of envelopes of rigid vertices. Either it contains edges $e=(x,z), e'=(z,y)$ as in item 2. of  Proposition \ref{SeeAclEqInLambda}, or it contains a surface type vertex $v$. In the first case, we get that the intersection of $\acl^{eq}(Ab)$ and $\acl^{eq}(Ac)$ contains the conjugacy class of a pair $(s_x, s_y)$ as in item 2. of Proposition \ref{SeeAclEqInLambda}. This conjugacy class does not lie in $\acl^{eq}(A)$ by Proposition \ref{NotInAcl}, which contradicts the independence of $b$ and $c$ over $A$. In the second case, there are elements $g,g'$ corresponding to non boundary parallel simple closed curves on the surface associated to $v$ such that the conjugacy class of $g$ lies in $\acl^{eq}(Ab)$ and that of $g'$ lies in $\acl^{eq}(Ac)$. But such elements fork over $A$ by Proposition \ref{NotInAcl}.

Thus we know that the intersection of any two of $\Gamma^i_{Ab}$ and $\Gamma^j_{Ac}$ must be a disjoint union of envelopes of rigid vertices. In this case, we consider the group $\hat{\F} = H * \F_A * H'$ where $H'$ is a copy of $H$, and the element $c'$ of $\F_A*H'$ corresponding to $c$ under the obvious isomorphism $H * \F_A \to \F_A * H'$. We can build a graph of group $\hat{\Lambda}$ for $\hat{\F}$ by "amalgamating" the graphs of groups $\tilde{\Lambda}^+_{Ab}$ and $\tilde{\Lambda}^+_{Ac}$ along $\Lambda_A$, in other words by adding to $\Lambda_A$ trivially stabilized edges ($\rk(H) + \rk(H')$ of them) according to how they are attached in $\tilde{\Lambda}^+_{Ab}$ and $\tilde{\Lambda}^+_{Ac}$ respectively, and associating them with bases for $H$ and $H'$ respectively. 

Note that in this graph of groups $\hat{\Lambda}$, the minimal subgraphs for $Ab$ and $Ac'$ intersect at most in a disjoint union of envelopes of rigid vertices, hence by Theorem \ref{DisjointnessImpliesIndependence}, we see that $b$ is independent from $c'$ over $A$. By Facts \ref{ForkTrans} and \ref{ForkAlg} we get that $b$ is independent 
from $c'$ over $acl(A)$. Now recall that we assumed that $b$ is independent from $c$ over $acl(A)$: by Fact \ref{AlgStat}, every type over an algebraically closed set is stationary and since $c$ and $c'$ 
have the same type over $acl(A)$, they have the same type over $acl(A)b$. Therefore, by homogeneity, 
they are in the same orbit under $\Aut(\hat{\F}/Ab)$. Since $\Aut(\hat{\F}/Ab)$ preserves $\F_A$, the automorphism sending $c'$ to $c$ and fixing $Ab$ sends the decomposition $\hat{\F} = H*\F_A*H'$ to a decomposition $\hat{\F} = H_1*\F_A*H'_1$ with $b \in H_1*\F_A$ and $c \in \F_A*H'_1$. By Kurosh's Theorem, this induces on $\F$ a decomposition $\F = \F_b * \F_{A} * \F_c$ for which $b \in \F_b * \F_{A}$ and $c \in \F_{A} * \F_c$.

Now this means that $\F_{Ab} = \F_b * \F_A$ while $\F_{Ac}= \F_A * \F_c$. But now we can build a JSJ decomposition for $\F$ by amalgamating the decompositions $\tilde{\Lambda}_{Ab}$ and $\tilde{\Lambda}_{Ac}$, and clearly in this decomposition the minimal subgraphs of $Ab$ and $Ac$ intersect at most in a disjoint union of envelopes of rigid vertices.  
\end{proof}

\providecommand{\bysame}{\leavevmode\hbox to3em{\hrulefill}\thinspace}
\providecommand{\MR}{\relax\ifhmode\unskip\space\fi MR }
\providecommand{\MRhref}[2]{%
  \href{http://www.ams.org/mathscinet-getitem?mr=#1}{#2}
}
\providecommand{\href}[2]{#2}

\end{document}